\documentclass[11pt]{article}
\usepackage{geometry}
\usepackage{color}
\usepackage{amssymb}
\usepackage{amsmath}
\usepackage{bbm}
\usepackage{psfrag}
\usepackage{color}
\usepackage{graphicx}
\usepackage{amsthm}

\newtheorem{theorem}{Theorem}[section] 

\newtheorem{definition}[theorem]{Definition}

\begin{document}
\title{Equivalence between estimates for quasi-commutators and commutators}
\author{Hans Zwart\\
University of Twente\\ Department of Applied Mathematics\\ P.O.\ Box 217, 7500 AE Enschede\\ The Netherlands}

\maketitle

\begin{abstract}
  In this short note we show that under some mild conditions on the space and the operators, an estimate for $\|Sf(A) - f(B)S\|$ needs only to be studied for invertible $S$ and $B$ equal to $A$. Thus estimates for a quasi-commutator can be derived from that for the commutator.
\end{abstract}

\section{Introduction}

Continuing the work of Peller \cite{Pell80},  Aleksandrov, Peller, Potapov, and Sukochev \cite{APPS11}, Rozendaal, Sukochev, and Tomskova \cite{RoST16} and others, we study the Lipschitz continuity of operator valued functions. A question that is studied in the above mentioned references is finding estimates for $\|S f(A) - f(B)S\|$ in terms of $\|S A - BS\|$. Here we show that under mild conditions this question needs only to be answered for $S$ invertible and $B=A$.

\section{Equivalence between commutator and quasi-commutator estimates}

To show the equivalence between estimates for the quasi- and for normal commutator problem, we have to assume that our class of spaces and functional calculus satisfies certain properties. We start with a property for the spaces.
\begin{definition}
Let ${\mathcal C}$ be a class of normed linear spaces. We say that ${\mathcal C}$ has the {\em stacking property} if it satisfies:
\begin{enumerate}
\item If $X_1, X_2 \in {\mathcal C}$, then $X_1 \oplus X_2 \in {\mathcal C}$. 
\item If $R$ is a bounded operator from $X_2$ to $X_1$ ($R \in {\mathcal L}(X_2,X_1)$), then
\begin{equation}
\label{eq:1}
  \left\| \left(\begin{array}{cc} 0 & R \\ 0 & 0 \end{array} \right) \right\|_{X_1 \oplus X_2} = \|R\|.
\end{equation}
\end{enumerate}
\end{definition}
Note that the last condition gives a relation between the norm of $X_1 \oplus X_2$ and that of $X_1$ and $X_2$. It could be replaced by the conditions
\[
  \left\|\left(\begin{array}{c} x_1 \\ 0 \end{array} \right)\right\|_{X_1 \oplus X_2} = \|x_1\| \mbox{ and }  \|x_2\| \leq \left\|\left(\begin{array}{c}  x_1 \\ x_2 \end{array} \right)\right\|_{X_1 \oplus X_2}.
\]
It is clear that the class of all Hilbert spaces has this property, but also the class of all ${\mathbb R}^n$, i.e., ${\mathcal C} = \{ {\mathbb R},\ldots, {\mathbb R}^n, \ldots \}$ possesses it. 

Associated with the class ${\mathcal C}$ with the stacking property and a scalar function $f$, we have a class of bounded linear operators ${\mathcal O}$ with the stacking property.
\begin{definition}
Let ${\mathcal C}$ be a class of normed linear spaces having the {\em stacking property}. We say that a class of bounded linear operators ${\mathcal O}$ has the {\em stacking property with respect to $f$} when the following three conditions hold
\begin{enumerate}
\item If $A_1, A_2 \in {\mathcal O}$, then $\left(\begin{array}{cc} A_1 & 0 \\ 0 & A_2 \end{array} \right) \in {\mathcal O}$,
\item For every $A \in {\mathcal O}$ $f(A)$ is well-defined, i.e., if $A \in {\mathcal O} \cap {\mathcal L}(X)$, then $f(A) \in {\mathcal L}(X)$, 
\item For every $A_1, A_2 \in {\mathcal O}$ there holds
\begin{equation}
\label{eq:6}
  f \left(\begin{array}{cc} A_1 & 0 \\ 0 & A_2 \end{array} \right) = \left(\begin{array}{cc} f(A_1) & 0 \\ 0 & f(A_2) \end{array} \right) .
\end{equation}
\end{enumerate}
\end{definition}

It is easy to see that the class of all self-adjoint operators (on Hilbert spaces) have the stacking property for every continuous $f$.

%
The following theorem shows that if the stacking property holds for ${\mathcal C}$ and ${\mathcal O}$, then the quasi- and standard commutator estimate are equivalent.
\begin{theorem}
\label{Tm1}
Let ${\mathcal C}$ be a class of normed linear spaces having the stacking property, and let ${\mathcal O}$ be a class of bounded linear operators having the stacking property with respect to $f$ (given). If for all $X \in {\mathcal C}$, $A \in {\mathcal O} \cap {\mathcal L}(X)$  and invertible $Q \in {\mathcal L}(X)$, there exists a $g_1 : [0,\infty) \mapsto [0,\infty)$ such that the following inequality holds
\begin{equation}
\label{eq:2}
  \|Q f(A) - f(A) Q\| \leq g_1(\|QA - AQ\|),
\end{equation}
then for all $X_1,X_2 \in {\mathcal C}$, $A_1 \in {\mathcal O} \cap {\mathcal L}(X_1)$, $A_2 \in {\mathcal O} \cap {\mathcal L}(X_2)$ and $S \in {\mathcal L}(X_2,X_1)$  there exists a $g_2 : [0,\infty) \mapsto [0,\infty)$ such that the following inequality holds
\begin{equation}
\label{eq:3}
  \|f(A_1) S- S f(A_2) \| \leq  g_2( \|A_1S - SA_2 \|).
\end{equation}
Furthermore, $g_2$ equals the $g_1$ for $Q= \left(\begin{smallmatrix} I &-S \\ 0 & I \end{smallmatrix} \right) $ and $A=  \left(\begin{smallmatrix} A _1&0 \\ 0 & A_2 \end{smallmatrix} \right) $.
\end{theorem}
\begin{proof}
For the $A_1,A_2$, and $S$ as in (\ref{eq:3}) we have the following equality
\begin{align*}
 \left(\begin{array}{cc} I &-S \\ 0 & I \end{array} \right) \left(\begin{array}{cc} f(A_1) &0 \\ 0 & f(A_2) \end{array} \right) - &\  
 \left(\begin{array}{cc} f(A_1) &0 \\ 0 & f(A_2) \end{array} \right) \left(\begin{array}{cc} I & -S \\ 0 & I \end{array} \right) \\
  =&\ \left(\begin{array}{cc} 0 & f(A_1)S-Sf(A_2) \\ 0 & 0 \end{array} \right).
\end{align*}
By our assumptions we can see the top line as $Qf(A)-f(A)Q$ with
\[
  Q= \left(\begin{array}{cc} I &-S \\ 0 & I \end{array} \right)  \mbox{ and } A=  \left(\begin{array}{cc} A _1&0 \\ 0 & A_2 \end{array} \right).
\]
Note that the first operator is invertible.
Using the assumptions in the theorem, we obtain
\begin{align*}
  \| f(A)S-Sf(B)\| =&\ \left\|  \left(\begin{array}{cc} 0 & f(A)S-Sf(B) \\ 0 & 0 \end{array} \right) \right\| \\
  =&\ \left\|  \left(\begin{array}{cc} I &-S \\ 0 & I \end{array} \right) \left(\begin{array}{cc} f(A) &0 \\ 0 & f(B) \end{array} \right) -  
 \left(\begin{array}{cc} f(A) &0 \\ 0 & f(B) \end{array} \right) \left(\begin{array}{cc} I & -S \\ 0 & I \end{array} \right) \right\| \\
 \leq &\ g_1\left( \left\|  \left(\begin{array}{cc} I &-S \\ 0 & I \end{array} \right) \left(\begin{array}{cc} A &0 \\ 0 & B \end{array} \right) -  
 \left(\begin{array}{cc} A &0 \\ 0 & B \end{array} \right) \left(\begin{array}{cc} I & -S \\ 0 & I \end{array} \right) \right\|\right) \\
 =&\ g_1\left( \left\|  \left(\begin{array}{cc} 0 & AS-SB \\ 0 & 0 \end{array} \right) \right\| \right) = g_1(\| AS-SB\| ),
\end{align*}
where we have used (\ref{eq:2}). Thus we have shown (\ref{eq:3}).
%
%
\end{proof}

From the theorem we see that if (for instance) $g_1$ can be chosen independently of $Q$, then $g_2$ is independent of $S$. 

As a consequence of our technique, we show next that if $g_1$ can be chosen independent of $Q$, then $f$ must be Lipschitz continuous.
\begin{theorem}
\label{Tm3}
Let ${\mathcal C}$ be a class of normed linear spaces having the stacking property, and let ${\mathcal O}$ be a class of operators having the stacking property with respect to a given $f$. Furthermore, let for all $X \in {\mathcal C}$, $A \in {\mathcal L}(X) \cap {\mathcal O}$  and invertible $Q \in {\mathcal L}(X)$ the following inequality holds
\begin{equation}
\label{eq:7-bis}
  \|Q f(A) - f(A) Q\| \leq g_1(\|QA - AQ\|).
\end{equation}
If $g_1$ can be chosen independently of $Q$, and if for some $A_1 \in {\mathcal O}$ and some 
 $\varepsilon \in {\mathbb C}$ we have that $A_1+ \varepsilon I \in {\mathcal O}$,  then the following inequality holds
\begin{equation}
\label{eq:10}
  \|f(A_1+\varepsilon I) -  f(A_1) \| \leq M|\varepsilon| 
\end{equation}
with $M$ independent of $\varepsilon$.
\end{theorem}
\begin{proof}
The following equalities are immediately 
\[
  \left(\begin{array}{cc} I & \varepsilon^{-1} I  \\ 0 & I \end{array} \right) \left(\begin{array}{cc} A_1 &0 \\ 0 & A_1+\varepsilon I \end{array} \right)  -  \left(\begin{array}{cc} A_1  & 0 \\ 0 & A_1 + \varepsilon I \end{array} \right) \left(\begin{array}{cc} I & \varepsilon^{-1} I  \\ 0 & I \end{array} \right) = \left(\begin{array}{cc} 0  & I \\ 0 &0 \end{array} \right).
\]
\begin{align*}
  \left(\begin{array}{cc} I & \varepsilon^{-1} I  \\ 0 & I \end{array} \right) \left(\begin{array}{cc} f(A_1) &0 \\ 0 & f(A_1+\varepsilon I )\end{array} \right)  - &\ \left(\begin{array}{cc} f(A_1)  & 0 \\ 0 & f(A_1 + \varepsilon I) \end{array} \right) \left(\begin{array}{cc} I & \varepsilon^{-1} I  \\ 0 & I \end{array} \right) = \\
  &\left(\begin{array}{cc} 0  & \varepsilon^{-1} (f(A_1+\varepsilon I )-f(A_1)) \\ 0 &0 \end{array} \right).
\end{align*}
So as in the proof of Theorem \ref{Tm1} we find that
\[
  \| \varepsilon^{-1} (f(A_1+\varepsilon I )-f(A_1)) \| \leq g_1(\|I\|) = g_1(1)
\]
Since by assumption $g_1$ is independent of $Q$, it is independent of $\varepsilon$. Hence inequality (\ref{eq:10}) is shown.\end{proof}

As a consequence of the theorem we see that if ${\mathbb R} \in {\mathcal C}$ and also ${\mathbb R} \in {\mathcal O}$, then $f$ is Lipschitz continuous. Furthermore, if $f$ is an entire function, ${\mathbb C} \in {\mathcal C}$, and  ${\mathbb C} \in {\mathcal O}$, then by the maximum modulus Theorem of complex analysis $f(z) = m z + f_0$.

Similarly as the above proof we find the following. When $A_1,A_2 \in {\mathcal O}$ commute, then
\[
  \|(A_1-A_2)^{-1} (f(A_2)-f(A_1)) \| \leq g_1(\|I\|),
\]
provided $A_1-A_2$ is invertible.
Hereby we use that $S = (A_1-A_2)^{-1} $ satisfies $A_1S - SA_2 =I$.

Choosing $S=I$ in (\ref{eq:3}) we see from Theorem \ref{Tm1} that the inequality (\ref{eq:2}) implies 
\[
  \|f(A_1) - f(A_2) \| \leq g_2(\|A_1 - A_2 \| )
\]
for all $A_1,A_2 \in {\mathcal O}$. If the class ${\mathcal O}$ is closed under similarity transformation, then the reserve holds as well.
\begin{theorem}
\label{Tm4}
Let ${\mathcal C}$ be a class of normed linear spaces having the stacking property, and let ${\mathcal O}$ be a class of operators having the stacking property with respect to a given $f$. Assume further that ${\mathcal O}$ is closed under similarity transformation, i.e., if $A \in {\mathcal L}(X)$ lies in ${\mathcal O}$, then also $QAQ^{-1} \in {\mathcal O}$ for every (boundedly) invertible $Q \in {\mathcal L}(X)$ and $f( QAQ^{-1}) = Q f(A) Q^{-1}$.

Under these assumptions we have that if for all $X \in {\mathcal C}$ and for all $A,B \in {\mathcal L}(X) \cap {\mathcal O}$ the following inequality holds
\begin{equation}
\label{eq:7}
  \| f(A) - f(B) \| \leq g_1(\|A - B\|),
\end{equation}
then for all $X_1,X_2 \in {\mathcal C}$, $A \in {\mathcal L}(X_1) \cap {\mathcal O}$, $B \in {\mathcal L}(X_2)\cap {\mathcal O}$ and $S \in {\mathcal L}(X_2,X_1)$ the following inequality holds
\begin{equation}
\label{eq:8}
  \|f(A_1) S- S f(A_2) \| \leq g_2( \|A_1S - SA_2 \|).
\end{equation}
\end{theorem}
\begin{proof} 
For the $A_1,A_2$, and $S$ in (\ref{eq:8}) we have the following equality
\begin{equation}
\label{eq:12}
  \left(\begin{array}{cc} I &-S \\ 0 & I \end{array} \right) \left(\begin{array}{cc} A_1 &0 \\ 0 & A_2 \end{array} \right) \left(\begin{array}{cc} I & S \\ 0 & I \end{array} \right) 
   = \left(\begin{array}{cc} A_1 & A_1S-SA_2 \\ 0 & A_2 \end{array} \right) =:B.
\end{equation}
By our assumption, we find
\begin{align}
\nonumber
  f(B)=&\  \left(\begin{array}{cc} I &-S \\ 0 & I \end{array} \right) \left(\begin{array}{cc} f(A_1) &0 \\ 0 & f(A_2) \end{array} \right) \left(\begin{array}{cc} I & S \\ 0 & I \end{array} \right) \\
  \label{eq:4}
  =&\ \left(\begin{array}{cc} f(A_1) & f(A_1)S-Sf(A_2) \\ 0 & f(A_2) \end{array} \right).
\end{align}
Applying inequality (\ref{eq:7}) with this $B$ and $A= \left( \begin{smallmatrix} A_1 &0 \\ 0 & A_2 \end{smallmatrix} \right)$ we obtain
\begin{align*}
  \| f(A_1)S-Sf(A_2)\| =&\ \left\|  \left(\begin{array}{cc} 0 & f(A_1)S-Sf(A_2) \\ 0 & 0 \end{array} \right) \right\| \\
  =&\ \left\| \left(\begin{array}{cc} f(A_1) & f(A_1)S-Sf(A_2) \\ 0 & f(A_2) \end{array} \right) - 
  \left(\begin{array}{cc} f(A_1) &0 \\ 0 & f(A_2) \end{array} \right) \right\| \\
  =&\ \|f(B) - f(A) \| \\
 \leq &\ g_1(\|B-A\|) \\
  =&\ g_1\left( \left\|  \left(\begin{array}{cc} 0 & A_1S-SA_2 \\ 0 & 0 \end{array} \right) \right\| \right) = g_1(\| A_1S-SA_2\| ),
\end{align*}
which proves the assertion.
\end{proof}

The above theorem indicates when operator Lipschitz continuity (inequality (\ref{eq:7})) is equivalent to a quasi-commutator property (inequality (\ref{eq:8})). In \cite{APPS11} the linear spaces (Hilbert spaces) and the operators (normal operators) have the stacking property, but the class of normal operators is not closed under similarity transformation, and so the above theorem does not apply. In \cite{RoST16} the class of spaces does not have the stacking property.

\end{document}